\pdfoutput=1
\documentclass[10pt,leqno]{amsart}

\usepackage{graphicx}
\usepackage{indentfirst,csquotes}

\usepackage{amssymb,amsthm,amsmath}
\usepackage{eucal}
\usepackage{xcolor,paralist,hyperref,titlesec,fancyhdr,etoolbox}

\usepackage{tikz}
\usetikzlibrary{arrows.meta,positioning}

\newtheorem{theorem}{Theorem}[section]

\newtheorem{proposition}[theorem]{Proposition}
\newtheorem{corollary}[theorem]{Corollary}

\theoremstyle{definition}
\newtheorem{definition}[theorem]{Definition}

\newtheorem{remark}[theorem]{Remark}

\makeatletter
\providecommand{\@secnumpunct}{.}   
\makeatother

\titleformat{\section}{\normalfont\large\bfseries\centering}{\thesection.}{0.5em}{}
\titlespacing*{\section}{0pt}{2.6ex plus 1ex minus .2ex}{1.6ex plus .2ex}

\hypersetup{ colorlinks=true, linkcolor=black, filecolor=black, urlcolor=black,
             citecolor=black }

\newcommand{\seq}{\Rightarrow}
\newcommand{\arr}{\rightarrow}
\newcommand{\Gt}{\mathsf{G3T}}
\newcommand{\Gz}{\mathsf{G0T}}
\newcommand{\Ng}{\mathsf{NgT}}
\newcommand{\LK}{\mathrm{LK}}

\newcommand{\wkL}{\mathrm{wkL}}
\newcommand{\wkR}{\mathrm{wkR}}
\newcommand{\ctrL}{\mathrm{ctrL}}
\newcommand{\oa}{\mathrm{oa}}
\newcommand{\ero}{\mathrm{end}}

\begin{document}

\title[Proof theory for sequent-style tableaux]{Proof theory for sequent-style tableaux: G0- and G3-style sequent calculi and full normalization}

\author[S.\ Cuconato]{Simone Cuconato}

\address{Department of Physics,
\newline \indent University of Calabria,
\newline \indent 87036 Rende (CS), Italy}
\email{simone.cuconato@unical.it
\newline \indent ORCID iD: \url{https://orcid.org/0000-0003-0277-9575}}



\date{\today}

\begin{abstract}
Sequent-style tableaux are a one-sided refutation calculus for classical propositional logic, in which each node of the refutation tree carries a finite block of formulae and the structural rules are absorbed into the data structure and the closure criterion. Building on the correspondence between this block calculus and the cut-free sequent calculus, and following the programme of Kamide and Negri, we recast the calculus as a structural-rule-free G3-style sequent calculus $\Gt$ with shared contexts, and we introduce a G0-style sequent calculus $\Gz$ with independent contexts, explicit weakening and contraction, generalized initial sequents, and a primitive explosion rule. A theorem establishing the equivalence between $\Gz$ and $\Gt$ is proved, and the cut-elimination theorem for $\Gz$ is obtained as a consequence. We then introduce a natural deduction system $\Ng$ with general elimination rules matching the left rules of $\Gz$, and we prove a full normalization theorem for $\Ng$. The proof is achieved by means of bidirectional translations between $\Gz$ and $\Ng$: normal derivations correspond to cut-free derivations, and full normal form to the discipline in which every major premiss of an elimination rule is an assumption. We also determine the reach of the formula-succedent fragment, which is shown to have no theorems, so that the equivalence of the three systems is one of consequence and not of theoremhood, and we show that classical logic is recovered on the succedent side, and recovered exactly, by adjoining the rule of indirect proof.
\end{abstract}

\maketitle

\section{Introduction}\label{sec:intro}

The proof theory of the family of Belnap--Dunn and intuitionistic logics has, in recent years, been organized around a precise correspondence between three kinds of deductive apparatus: G3-style sequent calculi, which are structural-rule-free and analytic; G0-style sequent calculi, which reintroduce the structural rules explicitly and split the contexts of the two-premiss rules; and natural deduction systems with general elimination rules, in which the eliminations discharge assumptions in the manner of disjunction elimination. Kamide and Negri \cite{KamideNegri2025} have shown, for Gurevich logic, Nelson logic and intuitionistic logic, that these three presentations are not merely equivalent in the coarse sense of proving the same theorems, but are tied together by translations that carry cut-free derivations to normal derivations and back, so that the cut-elimination theorem on the sequent side and the normalization theorem on the natural deduction side become two readings of one and the same structural fact. The technique of bidirectional translation between G0-style sequent calculi and natural deduction with general elimination rules, which goes back to von Plato and to Negri and von Plato \cite{vonPlato2001,NegriVonPlato2001,NegriVonPlatoJSL2001}, provides the fine-grained refinement of the relation between normalization and cut elimination on which their study rests.

The present paper carries out this programme for a calculus of a different provenance: the calculus of \emph{sequent-style tableaux}, introduced by the author \cite{Cuconato2025IM,cuconato2025metodi} as a one-sided refutation calculus for classical logic that grafts the sequent notation of Gentzen onto the branching architecture of analytic tableaux. In that calculus each node of the refutation tree carries not a single formula, as in the classical presentation of Smullyan \cite{Smullyan1968}, but a finite \emph{block} $\Pi$ of formulae, the counterpart of the antecedent of a sequent whose consequent has been absorbed, in negated form, into a single one-sided list. The passage from the sequent $\Gamma\seq\Delta$ to the block $\Pi=\Gamma\cup\lnot[\Delta]$ turns a two-sided calculus of synthesis into a one-sided calculus of analysis; the branch does not accumulate formulae as it descends, but transports the whole block, decomposing at each step the principal formula while inheriting the remaining context unchanged. The calculus is cut-free by construction, it enjoys the subformula property, and it stands in a bidirectional structural correspondence with the cut-free sequent calculus $\LK$ of Gentzen, a correspondence that the author has established in full \cite{Cuconato2025IM}. It is this last result that makes the sequent-style tableaux a natural object for the Kamide--Negri treatment: a calculus already known to coincide, on the plane of derivability, with cut-free $\LK$ invites the question whether the whole architecture of G0- and G3-style calculi and of general-elimination natural deduction can be erected upon it.

We answer this question in the affirmative for the propositional fragment. Our point of departure is the observation that the block calculus, read as a sequent calculus, is already in G3 style: its two-premiss rules share their context, its initial closed blocks are complementary pairs, and the structural rules of weakening and contraction are not primitive but absorbed, the former into the closure criterion and the latter into the set-theoretic reading of blocks. We make this reading explicit by defining a G3-style sequent calculus $\Gt$ whose sequents carry, besides the block, a possibly empty succedent slot, so that the purely refutational sequents $\Pi\seq$ recover the block calculus while the sequents $\Gamma\seq C$ with $C$ a formula give room for the introduction rules that the natural deduction correspondence requires. Alongside $\Gt$ we define a G0-style calculus $\Gz$ in which the contexts of the two-premiss rules are independent, weakening and contraction are explicit rules, the initial sequents are generalized to arbitrary formulae, and the explosion rule that governs the interaction of a formula with its negation is made primitive, in the sequent form that Kamide and Negri give it. We then prove a theorem establishing the equivalence of $\Gz$ and $\Gt$, and we obtain, as an immediate consequence, the cut-elimination theorem for $\Gz$.

The second half of the paper is devoted to natural deduction. We introduce a system $\Ng$ with general elimination rules for the connectives and for their negated forms, together with the matching introduction rules and the explosion rule, and we define, following von Plato, the notion of \emph{full normal form}: a derivation is fully normal when every major premiss of an elimination rule is an assumption. We then present two translations, one from fully normal derivations of $\Ng$ to derivations of $\Gz$, and one from derivations of $\Gz$ to fully normal derivations of $\Ng$, and we use them to prove the full normalization theorem for $\Ng$. The bidirectional character of the translation is essential: it is by passing to the sequent calculus, eliminating the cut, and translating back that an arbitrary derivation is brought to full normal form, so that the normalization theorem for $\Ng$ rests, in the end, on the cut-elimination theorem for $\Gz$.

Two choices of presentation deserve a word of comment. First, the negation of the block calculus is classical, and we treat it throughout as an \emph{involutive De Morgan negation}: it is governed by the double-negation rule $\lnot\lnot A\rightsquigarrow A$ and by the De Morgan rules for the connectives, exactly the apparatus that Kamide and Negri employ for the strong negation of the Belnap--Dunn family. This is not a restriction but a convenience: it keeps every negation rule in the analytic, subformula-respecting form that the translations require. It should be said at once where the classical strength of the system resides. It resides in the \emph{refutational} reading: the blocks that close are exactly the classically unsatisfiable finite sets of formulae, and on that fragment the involution, the explosion rule and the closure on complementary pairs deliver classical logic in full. The formula-succedent fragment is another matter, and we determine it exactly in Proposition~\ref{prop:noempty}: no sequent with empty antecedent is derivable, so that $\Ng$ has no theorems and the correspondence between the three systems is a correspondence of consequence rather than of theoremhood. Subsection~\ref{subsec:raa} shows that classical logic is restored on the succedent side, and restored exactly, by adjoining the rule of indirect proof. Second, we confine the development to the propositional level. The first-order block calculus adds the universal and existential rules, with their parameter and freshness conditions, and the whole apparatus of G0- and G3-style calculi and of general-elimination natural deduction extends to them along the lines of Kamide and Negri; we indicate the extension in a concluding remark (Remark~\ref{rem:firstorder}) but do not carry it out here, so as to keep the structural analysis unencumbered.

The paper is organized as follows. Section~\ref{sec:prelim} fixes the language and recalls the block calculus and its correspondence with cut-free $\LK$. Section~\ref{sec:sequents} introduces the G3- and G0-style sequent calculi $\Gt$ and $\Gz$, establishes their structural properties, proves the equivalence theorem, and derives the cut-elimination theorem for $\Gz$. Section~\ref{sec:nd} introduces the natural deduction system $\Ng$ with general elimination rules, defines full normal form, presents the two translations, and proves the full normalization theorem and the equivalence between the natural deduction and the sequent calculi. Section~\ref{sec:remarks} collects some concluding remarks, including the extension to first-order logic.

\section{The block calculus and its sequent counterpart}\label{sec:prelim}

\subsection{Language}

The formulae of the propositional language are built from a countable set of propositional variables $p,q,r,\dots$ by means of the connectives $\lnot$ (negation), $\land$ (conjunction), $\lor$ (disjunction) and $\arr$ (implication). We use $A,B,C,\dots$ for arbitrary formulae and, as is customary in the tableau tradition, we read the implication classically, so that $A\arr B$ behaves as $\lnot A\lor B$; this is reflected in the rules below, in which the implication is decomposed exactly as the corresponding disjunction. A \emph{literal} is a propositional variable or the negation of a propositional variable. The negation is taken as \emph{involutive}: the double negation $\lnot\lnot A$ is governed by a rule that identifies it, deductively, with $A$, and the negations of the compound formulae are governed by the De Morgan rules. This is the treatment that the strong negation of the Belnap--Dunn family receives in \cite{KamideNegri2025}, and we adopt it here because it keeps every negation rule analytic while yielding, in the presence of the closure on complementary pairs, the full classical logic.

The letter $\Gamma$, possibly decorated, denotes a finite multiset of formulae; the letter $\Pi$ denotes a finite \emph{set} of formulae, which we call a \emph{block}. We write $\lnot[\Delta]$ for the set $\{\lnot A\mid A\in\Delta\}$.

\subsection{The block calculus}\label{subsec:block}

We recall the block calculus of \cite{Cuconato2025IM} in its propositional form. A \emph{block} is a finite set of formulae; it is the datum carried by each node of a refutation tree. A block is \emph{closed} if it contains a complementary pair, that is, a formula $\theta$ together with its negation $\lnot\theta$; a closed block is marked with $\times$ and terminates its branch. The rules of the calculus decompose a distinguished \emph{principal} formula of a block into the block that lies below it, inheriting the rest of the block --- written $\Pi$ --- unchanged. Following Smullyan's classification, the one-premiss rules are of type $\alpha$ (conjunctive) and the two-premiss rules of type $\beta$ (disjunctive).

\begin{definition}[Block rules]\label{def:blockrules}
The rules of type $\alpha$ are
\[
\frac{\Pi,\lnot\lnot A}{\Pi,A}\ \lnot\lnot
\qquad
\frac{\Pi,A\land B}{\Pi,A,B}\ \land
\qquad
\frac{\Pi,\lnot(A\lor B)}{\Pi,\lnot A,\lnot B}\ \lnot\lor
\qquad
\frac{\Pi,\lnot(A\arr B)}{\Pi,A,\lnot B}\ \lnot\arr
\]
and the rules of type $\beta$ are
\[
\frac{\Pi,\lnot(A\land B)}{\Pi,\lnot A\ \mid\ \Pi,\lnot B}\ \lnot\land
\qquad
\frac{\Pi,A\lor B}{\Pi,A\ \mid\ \Pi,B}\ \lor
\qquad
\frac{\Pi,A\arr B}{\Pi,\lnot A\ \mid\ \Pi,B}\ \arr
\]
In each rule the principal formula, displayed to the right of $\Pi$, is removed and replaced by the formulae below the line; an $\alpha$ rule produces a single child block, a $\beta$ rule produces two child blocks. A \emph{tableau} for a block $\Pi$ is a finite tree, rooted at $\Pi$, grown by these rules; it is \emph{closed} if every one of its leaves is a closed block.
\end{definition}

The calculus is sound and complete for classical propositional logic in the following sense: a block $\Pi$ admits a closed tableau if and only if $\Pi$, read as the conjunction of its members, is unsatisfiable \cite{Cuconato2025IM}. Since the block associated with a sequent $\Gamma\seq\Delta$ is $\Pi=\Gamma\cup\lnot[\Delta]$, and $\Gamma\seq\Delta$ is valid precisely when $\Gamma\cup\lnot[\Delta]$ is unsatisfiable, the closure of the tableau for $\Pi$ is equivalent to the validity of the sequent. The structural rules are not primitive: contraction is absorbed by the set-theoretic nature of blocks, which record no multiplicities; weakening is absorbed by the closure criterion, since a block closes as soon as it contains a complementary pair, independently of the remaining formulae; and no cut rule is present, no rule introducing, on passing from a block to its children, a formula that is not a component of the principal one. The calculus is, in this precise sense, analytic.

\subsection{The correspondence with cut-free $\LK$}\label{subsec:corr}

The block calculus is not merely a notational variant of the tableau method: it stands in a bidirectional structural correspondence with the cut-free sequent calculus. We read a block $\Pi$ as the one-sided sequent $\Pi\seq$ with empty succedent, and a closed block $\{\theta,\lnot\theta\}$ as the sequent $\theta,\lnot\theta\seq$, derivable in $\LK$ from the axiom $\theta\seq\theta$. Under this reading the following holds \cite[Theorem~7.5]{Cuconato2025IM}: for $\Pi=\Gamma\cup\lnot[\Delta]$, there exists a closed tableau for $\Pi$ if and only if $\Gamma\seq\Delta$ is derivable in cut-free $\LK$. The translation carries a closed tableau to a cut-free derivation by inverting the orientation of the tree --- the root becomes the endsequent, the closed leaves become axioms, each $\beta$ bifurcation becomes a two-premiss rule, each $\alpha$ step a one-premiss rule --- and, in the converse direction, it carries a cut-free derivation to a closed tableau by the same inversion. No step of either translation employs the cut rule.

It is this correspondence that licenses the reformulation we undertake in the next section. A calculus that coincides, derivation by derivation, with cut-free $\LK$ is a calculus already in normal form, and the machinery of structural proof theory --- the distinction between G0- and G3-style presentations, the reduction of cut elimination to the interplay of the two, and the translation into natural deduction --- can be brought to bear on it directly. We turn to this now, taking care to present the sequent calculi in a form in which the block calculus is literally the fragment with empty succedent.

\section{G3- and G0-style sequent calculi and cut elimination}\label{sec:sequents}

We now transcribe the block calculus into sequent form. The design is governed by two requirements. On the one hand, the purely refutational sequents, those with empty succedent, must reproduce the block calculus exactly, so that the classical content established in \cite{Cuconato2025IM} is preserved without alteration. On the other hand, the sequents must carry a succedent slot, so that the introduction rules --- absent from a pure refutation calculus but indispensable for the correspondence with natural deduction --- find their place. We meet both requirements by taking sequents of the form $\Gamma\seq C$ in which $C$ is a formula or is empty; the block calculus is then the fragment in which $C$ is always empty.

\subsection{The calculus $\Gt$}\label{subsec:G3T}

\begin{definition}[Sequents]\label{def:sequent}
A \emph{sequent} is an expression $\Gamma\seq C$ where $\Gamma$ is a finite multiset of formulae and $C$ is either a formula or empty. We call $\Gamma$ the \emph{antecedent} and $C$ the \emph{succedent}. We write $L\vdash\Gamma\seq C$ to record that $\Gamma\seq C$ is derivable in the calculus $L$, and $L\vdash_n\Gamma\seq C$ to record that it is derivable with a derivation of height at most $n$. Two calculi $L_1,L_2$ are \emph{theorem-equivalent} if $L_1\vdash\Gamma\seq C$ holds exactly when $L_2\vdash\Gamma\seq C$ holds. A rule is \emph{admissible} in $L$ if, whenever its premisses are derivable in $L$, so is its conclusion; it is \emph{height-preserving admissible}, in symbols hp-admissible, if moreover the height of the derivation of the conclusion does not exceed the greatest of the heights of the derivations of the premisses; it is \emph{invertible} if the derivability of its conclusion entails that of each of its premisses.
\end{definition}

\begin{definition}[The calculus $\Gt$]\label{def:G3T}
Throughout the rules, $C$ denotes a formula or the empty succedent, and $p$ denotes a propositional variable. The calculus $\Gt$ consists of the following initial sequents and rules.

\smallskip
\noindent\emph{Initial sequents.}
\[
p,\Gamma\seq p\ \ \mathrm{(id)}
\qquad
\lnot p,\Gamma\seq\lnot p\ \ (\mathrm{id}^{\lnot})
\qquad
p,\lnot p,\Gamma\seq\ \ \mathrm{(cl)}
\]

\vspace{0.22cm}

\noindent\emph{Succedent weakening.}
\[
\frac{\Gamma\seq}{\Gamma\seq C}\ \wkR
\]

\[
\begin{array}{cc}
\displaystyle
\frac{A,B,\Gamma\seq C}{A\land B,\Gamma\seq C}\ \land\mathrm L
&
\displaystyle
\frac{A,\Gamma\seq C}{\lnot\lnot A,\Gamma\seq C}\ \lnot\lnot\mathrm L
\\[2ex]
\displaystyle
\frac{\lnot A,\lnot B,\Gamma\seq C}{\lnot(A\lor B),\Gamma\seq C}\ \lnot\lor\mathrm L
&
\displaystyle
\frac{A,\lnot B,\Gamma\seq C}{\lnot(A\arr B),\Gamma\seq C}\ \lnot\arr\mathrm L
\end{array}
\]

\vspace{0.22cm}

\noindent\emph{Left rules of type $\beta$ (two premisses, shared context).}
\[
\begin{array}{cc}
\displaystyle
\frac{\lnot A,\Gamma\seq C\quad \lnot B,\Gamma\seq C}
     {\lnot(A\land B),\Gamma\seq C}\ \lnot\land\mathrm L
&
\displaystyle
\frac{A,\Gamma\seq C\quad B,\Gamma\seq C}
     {A\lor B,\Gamma\seq C}\ \lor\mathrm L
\\[2ex]
\displaystyle
\frac{\lnot A,\Gamma\seq C\quad B,\Gamma\seq C}
     {A\arr B,\Gamma\seq C}\ \arr\mathrm L
&
\end{array}
\]

\vspace{0.22cm}

\noindent\emph{Right rules (introductions).}

\[
\begin{array}{cc}
\displaystyle
\frac{\Gamma\seq A\quad\Gamma\seq B}
     {\Gamma\seq A\land B}\ \land\mathrm R
&
\displaystyle
\frac{\Gamma\seq A}
     {\Gamma\seq A\lor B}\ \lor\mathrm R_1
\\[2ex]
\displaystyle
\frac{\Gamma\seq B}
     {\Gamma\seq A\lor B}\ \lor\mathrm R_2
&
\displaystyle
\frac{\Gamma\seq A}
     {\Gamma\seq\lnot\lnot A}\ \lnot\lnot\mathrm R
\\[2ex]
\displaystyle
\frac{\Gamma\seq\lnot A}
     {\Gamma\seq A\arr B}\ \arr\mathrm R_1
&
\displaystyle
\frac{\Gamma\seq B}
     {\Gamma\seq A\arr B}\ \arr\mathrm R_2
\\[2ex]
\displaystyle
\frac{\Gamma\seq\lnot A}
     {\Gamma\seq\lnot(A\land B)}\ \lnot\land\mathrm R_1
&
\displaystyle
\frac{\Gamma\seq\lnot B}
     {\Gamma\seq\lnot(A\land B)}\ \lnot\land\mathrm R_2
\\[2ex]
\displaystyle
\frac{\Gamma\seq\lnot A\quad\Gamma\seq\lnot B}
     {\Gamma\seq\lnot(A\lor B)}\ \lnot\lor\mathrm R
&
\displaystyle
\frac{\Gamma\seq A\quad\Gamma\seq\lnot B}
     {\Gamma\seq\lnot(A\arr B)}\ \lnot\arr\mathrm R
\end{array}
\]

\end{definition}

The reading of the rules is the intended one. Each $\alpha$ and $\beta$ left rule decomposes a signed formula of the antecedent in exactly the way the corresponding block rule decomposes it, carrying the succedent $C$ unaltered from conclusion to premisses; the two-premiss left rules share their context $\Gamma$, which is the defining trait of the G3 style. The right rules introduce a compound, or a negated compound, into the succedent, and it is here that the involutive treatment of negation shows its economy: a negated conjunction is introduced from a negated conjunct, a negated disjunction from the two negated disjuncts, a negated implication from the antecedent and the negated consequent, precisely as the De Morgan laws prescribe, so that no rule leaves the analytic form. The initial sequents are of three kinds: the identity on a variable and on a negated variable, which supply the succedent side, and the closure $p,\lnot p,\Gamma\seq$, which is the atomic complementary pair with empty succedent --- the sequent form of a closed block.

The connection with the block calculus is now immediate, and we record it before proceeding, since it is what anchors the whole construction to the calculus of \cite{Cuconato2025IM}.

\begin{proposition}[The refutational fragment is the block calculus]\label{prop:blockfragment}
For every finite set of formulae $\Pi$, the sequent $\Pi\seq$ with empty succedent is derivable in $\Gt$ if and only if the block $\Pi$ admits a closed tableau in the block calculus of Definition~\textup{\ref{def:blockrules}}.
\end{proposition}
\begin{proof}
Read a closed tableau for $\Pi$ upside down. Its closed leaves $\{\theta,\lnot\theta\}$ become, at the atomic level, instances of $\mathrm{(cl)}$, and at the compound level are reduced to the atomic case by the generalized closure of Proposition~\ref{prop:idcl} below; each $\alpha$ step becomes the corresponding one-premiss left rule with empty succedent, and each $\beta$ bifurcation the corresponding shared-context two-premiss left rule with empty succedent. Since no right rule and no succedent weakening can be applied when the succedent is and remains empty, and since the left rules with empty succedent are literally the block rules, the derivations of $\Pi\seq$ in $\Gt$ and the closed tableaux for $\Pi$ are in bijective correspondence. The equivalence follows.
\end{proof}

The succedent slot behaves quite differently, and it is well to record at once what it does and does not provide.

\begin{proposition}[The succedent fragment]\label{prop:noempty}
No sequent with empty antecedent is derivable in $\Gt$: for no formula $C$ is $\seq C$ derivable, and neither is the empty sequent $\seq$.
\end{proposition}
\begin{proof}
We show, by induction on the height of a derivation, that the antecedent of every derivable sequent is non-empty. At height zero the three initial sequents of $\Gt$ display, respectively, $p$, $\lnot p$ and the pair $p,\lnot p$ in the antecedent, so the claim holds. For the inductive step, a left rule carries its principal formula in the antecedent of its conclusion, whose antecedent is therefore non-empty whatever the premisses are; and a right rule, like $\wkR$, leaves the antecedent unchanged in passing from premisses to conclusion, so that the induction hypothesis applies.
\end{proof}

Two consequences deserve to be drawn immediately. The first is that the succedent slot, introduced to make room for the introduction rules, does not by itself yield a calculus for classical logic: neither $\seq p\lor\lnot p$ nor $\seq p\arr p$ is derivable, since $\lor\mathrm R_i$ and $\arr\mathrm R_i$ would require $\seq p$ or $\seq\lnot p$ as a premiss. The second is that the equivalence recorded in Theorem~\ref{thm:ndequiv} below is to be read as an equivalence of \emph{consequence}: what the three systems share is the relation between a multiset of assumptions and a conclusion, not a stock of theorems. What the introduction rules contribute is the constructive scaffolding on which the normalization theorem rests, and Subsection~\ref{subsec:raa} shows how the succedent side is brought up to classical logic.

\subsection{Structural properties of $\Gt$}\label{subsec:structural}

The calculus $\Gt$ has the good structural behaviour expected of a G3-style system: the initial sequents extend to arbitrary formulae, the structural rules of weakening and contraction are admissible without increase in height, and the negation rules of the G0 style, which we shall need in the equivalence theorem, are admissible. We establish these facts in the order in which they depend on one another.

\begin{proposition}[Generalized initial sequents and closure]\label{prop:idcl}
For every formula $A$ and every finite multiset $\Gamma$, the sequents
\[
A,\Gamma\seq A
\qquad\text{and}\qquad
A,\lnot A,\Gamma\seq
\]
are derivable in $\Gt$.
\end{proposition}
\begin{proof}
The proof is by simultaneous induction on the number of connectives of $A$; the case analysis is on the principal connective, and every reduction passes to formulae of strictly smaller weight, so that the induction hypothesis applies to them. If $A$ is a variable $p$, then $A,\Gamma\seq A$ is $\mathrm{(id)}$ and $A,\lnot A,\Gamma\seq$ is $\mathrm{(cl)}$. If $A$ is a negated variable $\lnot p$, then $A,\Gamma\seq A$ is $(\mathrm{id}^{\lnot})$, and $A,\lnot A,\Gamma\seq$, that is $\lnot p,\lnot\lnot p,\Gamma\seq$, is obtained from $\mathrm{(cl)}$ by one application of $\lnot\lnot\mathrm L$ to $\lnot\lnot p$, which yields $\lnot p,p,\Gamma\seq$.

For the inductive step we treat two representative cases and indicate the pattern for the others. Let first $A=B\land D$. The identity $B\land D,\Gamma\seq B\land D$ is derived by applying $\land\mathrm R$ to the two premisses $B\land D,\Gamma\seq B$ and $B\land D,\Gamma\seq D$; each is obtained by $\land\mathrm L$ from $B,D,\Gamma\seq B$, respectively $B,D,\Gamma\seq D$, which are available by the induction hypothesis for $B$ and for $D$, of smaller weight. The closure $B\land D,\lnot(B\land D),\Gamma\seq$ is derived by applying $\lnot\land\mathrm L$ to $\lnot(B\land D)$, reducing to $\lnot B,B\land D,\Gamma\seq$ and $\lnot D,B\land D,\Gamma\seq$; the first follows by $\land\mathrm L$ from $\lnot B,B,D,\Gamma\seq$, the second by $\land\mathrm L$ from $\lnot D,B,D,\Gamma\seq$, and both are instances of the closure available by the induction hypothesis for $B$ and for $D$. Let next $A=\lnot(B\land D)$, a negated conjunction. The identity $\lnot(B\land D),\Gamma\seq\lnot(B\land D)$ is derived by applying $\lnot\land\mathrm L$ to the antecedent occurrence, reducing to $\lnot B,\Gamma\seq\lnot(B\land D)$ and $\lnot D,\Gamma\seq\lnot(B\land D)$; the first follows by $\lnot\land\mathrm R_1$ from $\lnot B,\Gamma\seq\lnot B$, the second by $\lnot\land\mathrm R_2$ from $\lnot D,\Gamma\seq\lnot D$, both available by the induction hypothesis, since $\lnot B$ and $\lnot D$ have strictly smaller weight than $\lnot(B\land D)$. The closure $\lnot(B\land D),\lnot\lnot(B\land D),\Gamma\seq$ is obtained by $\lnot\lnot\mathrm L$ from $\lnot(B\land D),B\land D,\Gamma\seq$, which is the closure for $B\land D$, of smaller weight. The cases $A=B\lor D$, $A=B\arr D$ and their negations are analogous, using $\lor\mathrm R_i$, $\arr\mathrm R_i$, $\lnot\lor\mathrm R$, $\lnot\arr\mathrm R$ on the succedent side and the matching left rules on the antecedent side; the case $A=\lnot\lnot B$ reduces to $B$ by the involution rules $\lnot\lnot\mathrm L,\lnot\lnot\mathrm R$. In every case the succedent side uses a right rule and the closure side a left rule, and every appeal to the induction hypothesis is to a formula of strictly smaller weight.
\end{proof}

\begin{proposition}[Height-preserving admissibility of weakening]\label{prop:wk}
The rules of left and right weakening,
\[
\frac{\Gamma\seq C}{A,\Gamma\seq C}\ \wkL
\qquad
\frac{\Gamma\seq}{\Gamma\seq C}\ \wkR ,
\]
are hp-admissible in $\Gt$; the second is in fact a primitive rule.
\end{proposition}
\begin{proof}
Right weakening $\wkR$ is a rule of $\Gt$, so there is nothing to prove for it. For $\wkL$ we argue by induction on the height of the derivation of $\Gamma\seq C$. If $\Gamma\seq C$ is an initial sequent, then $A,\Gamma\seq C$ is again an initial sequent of the same kind, since the initial sequents allow an arbitrary antecedent context. If $\Gamma\seq C$ is the conclusion of a rule $R$, we apply the induction hypothesis to the premisses of $R$, adding $A$ to their antecedents --- which the shared-context format permits uniformly, and which respects the side-condition-free character of the propositional rules --- and then re-apply $R$. The height does not increase.
\end{proof}

\begin{proposition}[Invertibility]\label{prop:inv}
Every left rule of $\Gt$ is hp-invertible, and so are the right rules $\land\mathrm R$ and $\lnot\lor\mathrm R$; the one-premiss left rules and $\land\mathrm R$, $\lnot\lor\mathrm R$ are invertible with respect to each of their premisses.
\end{proposition}
\begin{proof}
The one-premiss left rules $\land\mathrm L$, $\lnot\lnot\mathrm L$, $\lnot\lor\mathrm L$, $\lnot\arr\mathrm L$ are invertible because their principal formula is, in the presence of Proposition~\ref{prop:idcl} and Proposition~\ref{prop:wk}, interderivable with the multiset of its active components; the formal argument is the standard induction on the height of the derivation of the conclusion, distinguishing whether the principal formula of the last rule is or is not the one under analysis, and using $\wkL$ in the base case. For the two-premiss left rules $\lnot\land\mathrm L$, $\lor\mathrm L$, $\arr\mathrm L$ and for $\land\mathrm R$, $\lnot\lor\mathrm R$ the same induction applies to each premiss separately. We omit the routine details, which follow \cite[Ch.~3]{NegriVonPlato2001}.
\end{proof}

\begin{proposition}[Height-preserving admissibility of contraction]\label{prop:ctr}
The rule of left contraction
\[
\frac{A,A,\Gamma\seq C}{A,\Gamma\seq C}\ \ctrL
\]
is hp-admissible in $\Gt$.
\end{proposition}
\begin{proof}
By induction on the height of the derivation of $A,A,\Gamma\seq C$, with the standard case distinction. If the sequent is initial, so is its contractum. If the last rule does not have one of the two displayed occurrences of $A$ as principal formula, the induction hypothesis is applied to its premisses and the rule is re-applied. If the last rule has one occurrence of $A$ as principal formula, then by the hp-invertibility of that rule (Proposition~\ref{prop:inv}) the other occurrence of $A$ may be replaced by its active components already in the premisses, at no cost in height; contracting on the components, whose height of derivation is smaller, by the induction hypothesis, and re-applying the rule, yields the contractum with no increase in height. The measure decreases because the components are of strictly smaller complexity, so the nested contractions terminate.
\end{proof}

The cut rule is admissible in $\Gt$. This is the counterpart, for the sequent presentation, of the cut-freeness of the block calculus; it can be obtained by the standard method, and we place it here because the admissibility of the explosion rule of the G0 style, which we establish next, rests upon it.

\begin{theorem}[Cut admissibility in $\Gt$]\label{thm:cutG3}
The cut rule
\[
\frac{\Gamma\seq A\quad A,\Delta\seq C}{\Gamma,\Delta\seq C}\ \mathrm{Cut}
\]
is admissible in $\Gt$.
\end{theorem}
\begin{proof}
The proof is by the standard double induction on the pair $(w,h)$, where $w$ is the weight of the cut formula $A$ --- the number of connectives in $A$ --- and $h$ is the sum of the heights of the derivations of the two premisses, following the method of \cite[Ch.~4]{NegriVonPlato2001} (see also \cite{TroelstraSchwichtenberg2000}). One considers a topmost cut and distinguishes three configurations. If one premiss is an initial sequent, the cut is removed outright. Against $\mathrm{(id)}$ or $(\mathrm{id}^{\lnot})$ the conclusion follows by left weakening (Proposition~\ref{prop:wk}). The only configuration proper to the refutational reading is a cut whose right premiss is the closure axiom $\mathrm{(cl)}$, say $p,\lnot p,\Delta\seq$ with cut formula $p$ and left premiss $\Gamma\seq p$: here the conclusion $\lnot p,\Gamma,\Delta\seq$ is obtained by the height-preserving transformation that, throughout the derivation of $\Gamma\seq p$, replaces the succedent occurrence of $p$ by the antecedent formula $\lnot p$ with empty succedent, turning every $\mathrm{(id)}$ leaf concluding $p$ into a $\mathrm{(cl)}$ leaf and every application of $\wkR$ into an application of $\wkL$, and then weakening in $\Delta$; the case with cut formula $\lnot p$ is symmetric. If the cut formula is not principal in one of the premisses, the cut is permuted upward past the last rule of that premiss, which strictly decreases $h$ at unchanged $w$; the invertibility of the rules (Proposition~\ref{prop:inv}) and the hp-admissibility of the structural rules guarantee that the permutation is licit, and the side-condition-free character of the propositional rules that it meets no obstruction. If the cut formula is principal in both premisses, the cut is reduced to cuts on the immediate components, of strictly smaller weight $w$: for $A=B\land D$ the cut between $\land\mathrm R$ and $\land\mathrm L$ reduces to a cut on $B$ and one on $D$; for $A=B\lor D$ between $\lor\mathrm R_i$ and $\lor\mathrm L$ to a cut on $B$ or on $D$; for $A=B\arr D$, treated as the material $\lnot B\lor D$, between $\arr\mathrm R_i$ and $\arr\mathrm L$ to a cut on $\lnot B$ or on $D$; for $A=\lnot\lnot B$ between $\lnot\lnot\mathrm R$ and $\lnot\lnot\mathrm L$ to a cut on $B$; and for the negated compounds $A=\lnot(B\land D),\lnot(B\lor D),\lnot(B\arr D)$ between the corresponding De Morgan right and left rules to cuts on $\lnot B$, $\lnot D$, or their pair, in every case of weight strictly below $w$. Since each reduction lowers $(w,h)$ in the lexicographic order, and a derivation contains finitely many cuts, all cuts are eliminated.
\end{proof}

We may now establish the admissibility of the explosion rule of the G0 style, which is the analogue, for the involutive negation of the block calculus, of the explosion rule that Kamide and Negri admit in their G3-style calculi \cite[Prop.~2.11]{KamideNegri2025}. It will be used in the equivalence theorem.

\begin{proposition}[Admissibility of explosion]\label{prop:ex}
The rule
\[
\frac{\Gamma\seq\lnot A\quad\Delta\seq A}{\Gamma,\Delta\seq C}\ \lnot\mathrm{EX}
\]
is admissible in $\Gt$.
\end{proposition}
\begin{proof}
Assume $\Gt\vdash\Gamma\seq\lnot A$ and $\Gt\vdash\Delta\seq A$. By the generalized closure of Proposition~\ref{prop:idcl}, the sequent $A,\lnot A\seq$ is derivable. A cut on $A$ between $\Delta\seq A$ and $A,\lnot A\seq$ yields $\lnot A,\Delta\seq$; a cut on $\lnot A$ between $\Gamma\seq\lnot A$ and $\lnot A,\Delta\seq$ yields $\Gamma,\Delta\seq$; and one application of $\wkR$ delivers $\Gamma,\Delta\seq C$. Both cuts are admissible by Theorem~\ref{thm:cutG3}, whose proof makes no appeal to the present rule.
\end{proof}

\subsection{The calculus $\Gz$}\label{subsec:G0T}

The G0-style calculus is obtained from $\Gt$ by the transformation that Kamide and Negri apply to their G3-style calculi \cite[Def.~2.3]{KamideNegri2025}: the two-premiss rules are given independent contexts, the structural rules of weakening and contraction are made primitive, the initial sequents are generalized to arbitrary formulae, and the explosion rule $\lnot\mathrm{EX}$ --- the sequent transcription of the interaction of a formula with its negation, in the manner of Gentzen's explosion \cite{Gentzen1935} --- is adopted as primitive.

\begin{definition}[The calculus $\Gz$]\label{def:G0T}
The calculus $\Gz$ consists of the following initial sequents and rules, where again $C$ is a formula or empty.

\smallskip
\noindent\emph{Initial sequents} (generalized): for every formula $A$,
\[
A\seq A\ \ (\mathrm{id}^{s})
\qquad
A,\lnot A\seq\ \ (\mathrm{cl}^{s}).
\]

\noindent\emph{Structural rules.}
\[
\frac{\Gamma\seq C}{A,\Gamma\seq C}\ \wkL
\qquad
\frac{A,A,\Gamma\seq C}{A,\Gamma\seq C}\ \ctrL
\qquad
\frac{\Gamma\seq}{\Gamma\seq C}\ \wkR
\]

\vspace{0.22cm}

\noindent\emph{Left rules of type $\alpha$}: as in $\Gt$ ($\land\mathrm L$, $\lnot\lnot\mathrm L$, $\lnot\lor\mathrm L$, $\lnot\arr\mathrm L$).

\vspace{0.22cm}

\smallskip
\noindent\emph{Left rules of type $\beta$ (independent contexts).}
\[
\begin{array}{cc}
\displaystyle
\frac{\lnot A,\Gamma\seq C\quad \lnot B,\Delta\seq C}
     {\lnot(A\land B),\Gamma,\Delta\seq C}\ \lnot\land\mathrm L^{s}
&
\displaystyle
\frac{A,\Gamma\seq C\quad B,\Delta\seq C}
     {A\lor B,\Gamma,\Delta\seq C}\ \lor\mathrm L^{s}
\\[2ex]
\displaystyle
\frac{\lnot A,\Gamma\seq C\quad B,\Delta\seq C}
     {A\arr B,\Gamma,\Delta\seq C}\ \arr\mathrm L^{s}
&
\end{array}
\]

\vspace{0.22cm}

\noindent\emph{Right rules} (with independent contexts where two premisses occur).
\[
\frac{\Gamma\seq A\quad\Delta\seq B}{\Gamma,\Delta\seq A\land B}\ \land\mathrm R^{s}
\qquad
\frac{\Gamma\seq\lnot A\quad\Delta\seq\lnot B}{\Gamma,\Delta\seq\lnot(A\lor B)}\ \lnot\lor\mathrm R^{s}
\qquad
\frac{\Gamma\seq A\quad\Delta\seq\lnot B}{\Gamma,\Delta\seq\lnot(A\arr B)}\ \lnot\arr\mathrm R^{s}
\]
together with the one-premiss right rules $\lor\mathrm R_1$, $\lor\mathrm R_2$, $\arr\mathrm R_1$, $\arr\mathrm R_2$, $\lnot\lnot\mathrm R$, $\lnot\land\mathrm R_1$, $\lnot\land\mathrm R_2$ exactly as in $\Gt$.

\vspace{0.22cm}

\smallskip
\noindent\emph{Explosion rule.}
\[
\frac{\Gamma\seq\lnot A\quad\Delta\seq A}{\Gamma,\Delta\seq C}\ \lnot\mathrm{EX}
\]
\end{definition}

The G0-style calculi are the ones on which the correspondence with natural deduction turns, because the independent contexts of their two-premiss rules match the way general elimination rules combine their subderivations, and because the explosion rule $\lnot\mathrm{EX}$ is the exact sequent counterpart of the natural deduction rule of explosion, while the De Morgan right rules are the sequent counterparts of the introduction rules for negated formulae. We shall exploit this in Section~\ref{sec:nd}; for the moment we establish the relation between $\Gz$ and $\Gt$.

\subsection{Equivalence and cut elimination}\label{subsec:equiv}

\begin{proposition}[Derivability of the $\Gt$ rules in $\Gz$]\label{prop:G3inG0}
Each initial sequent and each rule of $\Gt$ is derivable, or admissible, in $\Gz$. In particular the shared-context two-premiss rules, the atomic initial sequents, and the succedent-weakening rule of $\Gt$ are available in $\Gz$.
\end{proposition}
\begin{proof}
The atomic initial sequents $\mathrm{(id)}$, $(\mathrm{id}^{\lnot})$ and $\mathrm{(cl)}$ are instances of the generalized $(\mathrm{id}^{s})$ and $(\mathrm{cl}^{s})$ followed, where a nonempty context is required, by $\wkL$. The rule $\wkR$ is primitive in $\Gz$. The one-premiss left rules and the one-premiss right rules are shared verbatim by the two calculi. It remains to derive the shared-context two-premiss rules of $\Gt$ from the independent-context rules of $\Gz$. We show the case of $\lor\mathrm L$; the cases of $\lnot\land\mathrm L$, $\arr\mathrm L$, $\land\mathrm R$, $\lnot\lor\mathrm R$, $\lnot\arr\mathrm R$ are identical in pattern. Given $\Gz$-derivations of $A,\Gamma\seq C$ and $B,\Gamma\seq C$, an application of $\lor\mathrm L^{s}$ yields $A\lor B,\Gamma,\Gamma\seq C$, and repeated applications of $\ctrL$ contract the doubled context $\Gamma$ to a single copy, giving $A\lor B,\Gamma\seq C$, which is the conclusion of $\lor\mathrm L$. Thus every $\Gt$-derivation is transformed, step by step, into a $\Gz$-derivation of the same endsequent.
\end{proof}

\begin{theorem}[Equivalence of $\Gz$ and $\Gt$]\label{thm:equiv}
The calculi $\Gz$ and $\Gt$ are theorem-equivalent: for every sequent $\Gamma\seq C$,
\[
\Gz\vdash\Gamma\seq C\quad\text{if and only if}\quad\Gt\vdash\Gamma\seq C.
\]
\end{theorem}
\begin{proof}
The implication from $\Gt$ to $\Gz$ is Proposition~\ref{prop:G3inG0}. For the converse we show that every initial sequent and every rule of $\Gz$ is derivable or admissible in $\Gt$, whence a $\Gz$-derivation is transformed into a $\Gt$-derivation of the same endsequent by translating it from the leaves downward.

The generalized initial sequents $(\mathrm{id}^{s})$ and $(\mathrm{cl}^{s})$ are derivable in $\Gt$ by Proposition~\ref{prop:idcl}. The structural rules $\wkL$ and $\ctrL$ are hp-admissible in $\Gt$ by Propositions~\ref{prop:wk} and~\ref{prop:ctr}, and $\wkR$ is primitive. The one-premiss rules are shared. The explosion rule $\lnot\mathrm{EX}$ is admissible in $\Gt$ by Proposition~\ref{prop:ex}. It remains to treat the independent-context two-premiss rules. We show $\lor\mathrm L^{s}$; the others are identical in pattern. Given $\Gt$-derivations of $A,\Gamma\seq C$ and $B,\Delta\seq C$, apply $\wkL$ (Proposition~\ref{prop:wk}) to obtain $A,\Gamma,\Delta\seq C$ and $B,\Gamma,\Delta\seq C$, and then the shared-context rule $\lor\mathrm L$ to obtain $A\lor B,\Gamma,\Delta\seq C$, which is the conclusion of $\lor\mathrm L^{s}$. Since each $\Gz$-rule is thus reproduced in $\Gt$, the translation carries a $\Gz$-derivation to a $\Gt$-derivation of the same endsequent, and the equivalence is established.
\end{proof}

The equivalence between the two styles has, as its principal dividend, the cut-elimination theorem for the G0-style calculus. Since cut is admissible in $\Gt$ (Theorem~\ref{thm:cutG3}), it transfers to $\Gz$ through the equivalence.

\begin{theorem}[Cut elimination for $\Gz$]\label{thm:cutG0}
The cut rule is admissible in $\Gz$.
\end{theorem}
\begin{proof}
Suppose $\Gz\vdash\Gamma\seq A$ and $\Gz\vdash A,\Delta\seq C$. By Theorem~\ref{thm:equiv}, $\Gt\vdash\Gamma\seq A$ and $\Gt\vdash A,\Delta\seq C$. By Theorem~\ref{thm:cutG3}, cut is admissible in $\Gt$, so $\Gt\vdash\Gamma,\Delta\seq C$. By Theorem~\ref{thm:equiv} once more, $\Gz\vdash\Gamma,\Delta\seq C$, which was to be shown.
\end{proof}

The route just followed is the one that Kamide and Negri take for their G0-style calculi \cite[Thm.~2.13]{KamideNegri2025}: rather than eliminate cut directly in the G0 calculus, where the independent contexts and the explicit structural rules multiply the cases, one transfers the derivation to the structural-rule-free G3 calculus, eliminates the cut there, and transfers back. The equivalence theorem is thus not an end in itself but the instrument by which the cut-elimination theorem for $\Gz$ is obtained.

\begin{corollary}\label{cor:noemptyG0}
No sequent with empty antecedent is derivable in $\Gz$.
\end{corollary}
\begin{proof}
Immediate from Theorem~\ref{thm:equiv} and Proposition~\ref{prop:noempty}.
\end{proof}

\subsection{Recovering classical logic on the succedent side}\label{subsec:raa}

By Proposition~\ref{prop:noempty} the calculi of this section leave the succedent side without theorems, and the reason is structural: no rule passes a formula from the succedent to the antecedent, so a derivation starting from an empty antecedent can never reach an initial sequent. The classical rule of indirect proof does exactly that, and adjoining it repairs the deficiency in one step. Write $\Gt^{c}$ for the calculus obtained from $\Gt$ by adding
\[
\frac{\lnot A,\Gamma\seq}{\Gamma\seq A}\ \mathrm{RAA}
\]
The rule is not admissible in $\Gt$: were it admissible, $\seq p\lor\lnot p$ would be derivable, against Proposition~\ref{prop:noempty}. Its addition leaves the refutational fragment untouched, since its conclusion is never a sequent with empty succedent, so that Proposition~\ref{prop:blockfragment} holds of $\Gt^{c}$ as it stands.

\begin{proposition}[Adequacy of $\Gt^{c}$]\label{prop:raa}
For every finite multiset $\Gamma$ and every formula $A$, the sequent $\Gamma\seq A$ is derivable in $\Gt^{c}$ if and only if $A$ is a classical consequence of $\Gamma$.
\end{proposition}
\begin{proof}
Read $\Gamma\seq C$ as $\bigwedge\Gamma\models C$ and $\Gamma\seq$ as the unsatisfiability of $\Gamma$. Soundness is then checked rule by rule: the initial sequents and the left rules are sound in the usual way; each right rule introduces a compound, or a negated compound, whose classical truth condition follows from those of its premisses; $\wkR$ is sound because an unsatisfiable antecedent entails everything; and $\mathrm{RAA}$ is sound because the unsatisfiability of $\Gamma\cup\{\lnot A\}$ gives $\bigwedge\Gamma\models A$. For completeness, suppose $\bigwedge\Gamma\models A$. Then $\Gamma\cup\{\lnot A\}$ is classically unsatisfiable, so by the adequacy of the block calculus \cite{Cuconato2025IM} it admits a closed tableau, and by Proposition~\ref{prop:blockfragment} the corresponding sequent with empty succedent is derivable in $\Gt$; the multiplicities that the passage from a set to a multiset requires are restored by Proposition~\ref{prop:wk}, giving $\lnot A,\Gamma\seq$. One application of $\mathrm{RAA}$ yields $\Gamma\seq A$.
\end{proof}

In particular $\seq p\lor\lnot p$ is derivable in $\Gt^{c}$: the block $\{\lnot(p\lor\lnot p)\}$ closes, by one application of $\lnot\lor$ followed by one of $\lnot\lnot$, so that $\lnot(p\lor\lnot p)\seq$ is derivable in $\Gt$, and $\mathrm{RAA}$ concludes. The law of excluded middle is thus derived, and not assumed, exactly as one would wish; but it is derived in $\Gt^{c}$, and not in $\Gt$.

\begin{remark}\label{rem:raa}
Two comments on the extension. First, cut is admissible in $\Gt^{c}$, and for a reason that costs nothing: the premisses of a cut are classically valid consequences, hence so is its conclusion, which is therefore derivable --- by Proposition~\ref{prop:raa} when the conclusion has a formula in the succedent, and by the adequacy of the block calculus together with Proposition~\ref{prop:blockfragment} when it has none. The syntactic argument of Theorem~\ref{thm:cutG3} concerns $\Gt$ and is unaffected by the extension. Second, the natural deduction counterpart of $\mathrm{RAA}$ is the classical reductio, which discharges $\lnot A$ in a derivation of $\bot$; adjoining it to the system $\Ng$ of the next section yields a natural deduction system for classical logic, but normalization for general elimination rules in the presence of a classical rule calls for the apparatus of von Plato and Siders \cite{vonPlatoSiders2012}, and we do not pursue it here. The results of Section~\ref{sec:nd} concern $\Ng$ as defined in Definition~\ref{def:NgT}.
\end{remark}

\section{Natural deduction with general elimination rules and full normalization}\label{sec:nd}

We now introduce a natural deduction system for the logic of the block calculus, and we prove for it a full normalization theorem. The system is of the kind studied by von Plato and by Negri and von Plato \cite{vonPlato2001,NegriVonPlato2001,NegriVonPlatoJSL2001}, and whose general elimination rules trace back to Schroeder-Heister \cite{SchroederHeister1984}: its elimination rules are \emph{general}, in the sense that each of them, like the familiar rule of disjunction elimination, has a major premiss, a number of minor derivations in which assumptions are discharged, and a conclusion of arbitrary form. The point of the general format is that it makes the natural deduction rules match, one for one, the left rules of the G0-style sequent calculus, so that the two systems can be translated into each other; and it is by such a translation, run in both directions and passing through the cut-elimination theorem of the previous section, that we obtain the normalization theorem.

\subsection{The system $\Ng$}\label{subsec:NgT}

We adopt the standard conventions of natural deduction. A \emph{derivation} is a finite tree of formulae grown by the rules below; its topmost formulae are \emph{assumptions}, each either open or discharged by an application of a rule further down. An expression $[A]$ marks an assumption available for discharge; when a rule discharges it, we write the discharge above the inference. For a derivation $\mathcal D$ we write $\oa(\mathcal D)$ for the multiset of its open assumptions and $\ero(\mathcal D)$ for its end-formula. A formula $A$ is \emph{provable} from $\Gamma$ if there is a derivation with $\oa(\mathcal D)=\Gamma$ and $\ero(\mathcal D)=A$.

\begin{definition}[The system $\Ng$]\label{def:NgT}
The system $\Ng$ consists of the following introduction and elimination rules. As before, the implication is read materially, so that $A\arr B$ and $\lnot(A\arr B)$ are governed exactly as $\lnot A\lor B$ and $A\land\lnot B$ respectively.

\smallskip
\noindent\emph{Introduction rules.}
\[
\begin{array}{ccc}
\displaystyle
\frac{A\quad B}{A\land B}\ \land\mathrm I
&
\displaystyle
\frac{A}{A\lor B}\ \lor\mathrm I_1
&
\displaystyle
\frac{B}{A\lor B}\ \lor\mathrm I_2
\end{array}
\]

\[
\begin{array}{ccc}
\displaystyle
\frac{\lnot A}{A\arr B}\ \arr\mathrm I_1
&
\displaystyle
\frac{B}{A\arr B}\ \arr\mathrm I_2
&
\displaystyle
\frac{A}{\lnot\lnot A}\ \lnot\lnot\mathrm I
\end{array}
\]

\[
\begin{array}{cccc}
\displaystyle
\frac{\lnot A}{\lnot(A\land B)}\ \lnot\land\mathrm I_1
&
\displaystyle
\frac{\lnot B}{\lnot(A\land B)}\ \lnot\land\mathrm I_2
&
\displaystyle
\frac{\lnot A\quad\lnot B}{\lnot(A\lor B)}\ \lnot\lor\mathrm I
&
\displaystyle
\frac{A\quad\lnot B}{\lnot(A\arr B)}\ \lnot\arr\mathrm I
\end{array}
\]

\vspace{0.22cm}

\noindent\emph{General elimination rules.}
\[
\frac{A\land B\quad \begin{array}{c}[A,B]\\ \vdots\\ C\end{array}}{C}\ \land\mathrm E
\qquad
\frac{A\lor B\quad \begin{array}{c}[A]\\ \vdots\\ C\end{array}\quad \begin{array}{c}[B]\\ \vdots\\ C\end{array}}{C}\ \lor\mathrm E
\qquad
\frac{A\arr B\quad \begin{array}{c}[\lnot A]\\ \vdots\\ C\end{array}\quad \begin{array}{c}[B]\\ \vdots\\ C\end{array}}{C}\ \arr\mathrm E
\]
\[
\frac{\lnot\lnot A\quad \begin{array}{c}[A]\\ \vdots\\ C\end{array}}{C}\ \lnot\lnot\mathrm E
\qquad
\frac{\lnot(A\land B)\quad \begin{array}{c}[\lnot A]\\ \vdots\\ C\end{array}\quad \begin{array}{c}[\lnot B]\\ \vdots\\ C\end{array}}{C}\ \lnot\land\mathrm E
\]
\[
\frac{\lnot(A\lor B)\quad \begin{array}{c}[\lnot A,\lnot B]\\ \vdots\\ C\end{array}}{C}\ \lnot\lor\mathrm E
\qquad
\frac{\lnot(A\arr B)\quad \begin{array}{c}[A,\lnot B]\\ \vdots\\ C\end{array}}{C}\ \lnot\arr\mathrm E
\]

\noindent\emph{Explosion.}
\[
\frac{\lnot A\quad A}{C}\ \mathrm{Exp}
\]
\end{definition}

The rules $\land\mathrm E$, $\lor\mathrm E$, $\arr\mathrm E$, $\lnot\lnot\mathrm E$, $\lnot\land\mathrm E$, $\lnot\lor\mathrm E$, $\lnot\arr\mathrm E$ are the general elimination rules; in each, the leftmost premiss is the \emph{major} premiss, and the derivations from bracketed assumptions are the \emph{minor} derivations. In $\mathrm{Exp}$ both premisses $\lnot A$ and $A$ count as major, in accordance with the treatment of explosion in \cite[Rem.~3.5]{KamideNegri2025}. It is characteristic of the general format that the elimination of a connective does not project onto a distinguished component --- as the special elimination $\frac{A\land B}{A}$ would --- but discharges the components as assumptions of a minor derivation towards an arbitrary conclusion $C$; the special eliminations are recovered as instances in which $C$ is the component and the minor derivation is the assumption itself. The introduction rules and the explosion rule are exactly the natural deduction reading of the right rules and of the closure of the sequent calculus of Section~\ref{sec:sequents}: to each right rule corresponds an introduction, and to the closure axiom, read as the derivation of a contradiction from a complementary pair, corresponds $\mathrm{Exp}$.

Following von Plato and Negri, we single out the derivations in which no elimination is applied to a compound formula that has just been introduced.

\begin{definition}[Full normal form]\label{def:fnf}
A derivation in $\Ng$ is in \emph{full normal form}, or is \emph{fully normal}, if every major premiss of an elimination rule --- including both premisses of every application of $\mathrm{Exp}$ --- is an assumption, open or discharged.
\end{definition}

The notion of full normal form for natural deduction with general elimination rules is due to von Plato \cite{vonPlato2001,NegriVonPlato2001}; the demand that major premisses be assumptions is the general-elimination counterpart of the demand, in Prawitz's normalization \cite{Prawitz1965}, that no major premiss of an elimination be the conclusion of an introduction. As in \cite{KamideNegri2025}, we use the symbol $\bot$ as an abbreviation of the formula $\lnot p\land p$ for a fixed propositional variable $p$, and we interpret a derivation of a sequent $\Gamma\seq$ with empty succedent as a derivation $\mathcal D$ in $\Ng$ with $\oa(\mathcal D)=\Gamma$ and $\ero(\mathcal D)=\bot$. A \emph{multiset reduct} of a multiset $\Gamma$ is a multiset obtained from $\Gamma$ by lowering some of its multiplicities, the operation that corresponds, on the side of natural deduction, to the multiple discharge of assumptions \cite[Def.~3.7]{KamideNegri2025}.

We record a fact that will be used to translate the succedent-weakening rule.

\begin{proposition}[Derivability of the falsum rule]\label{prop:falsum}
The rule $\dfrac{\bot}{C}\ \bot\mathrm E$ is derivable in $\Ng$, and the derivation obtained is fully normal whenever its premiss is the conclusion of a fully normal derivation whose eliminations have assumptions as major premisses.
\end{proposition}
\begin{proof}
Let $\bot=\lnot p\land p$. From the major premiss $\lnot p\land p$, taken as an assumption, one application of $\land\mathrm E$ discharging $[\lnot p, p]$ towards the conclusion $C$, whose minor derivation is the single application of $\mathrm{Exp}$ to the two assumptions $\lnot p$ and $p$, yields $C$:
\[
\frac{\lnot p\land p\quad \begin{array}{c}[\lnot p,\ p]\\[2pt] \dfrac{\lnot p\quad p}{C}\ \mathrm{Exp}\end{array}}{C}\ \land\mathrm E .
\]
The major premiss of $\land\mathrm E$ is the assumption $\lnot p\land p$, and both premisses of $\mathrm{Exp}$ are the assumptions $\lnot p$ and $p$, so the derivation is fully normal; grafting it onto a fully normal derivation of $\bot$ whose own eliminations have assumptions as major premisses preserves full normality.
\end{proof}

\subsection{From natural deduction to sequent calculus}\label{subsec:ndtosc}

We first show that a derivation of $\Ng$ can be simulated in $\Gz$. The translation follows the last rule of the derivation; we display, for each rule, the first step in sequent notation and continue the translation on the premisses. When the derivation is fully normal, the simulation uses no cut; a non-normal derivation is simulated with cuts, which the cut-elimination theorem of Section~\ref{sec:sequents} then removes.

\begin{theorem}[Simulation of $\Ng$ in $\Gz$]\label{thm:ndtosc}
If $\mathcal D$ is a derivation in $\Ng$ with $\oa(\mathcal D)=\Gamma$ and $\ero(\mathcal D)=C$, then $\Gz\vdash\Gamma\seq C$. If $\mathcal D$ is fully normal, the simulating derivation uses no cut.
\end{theorem}
\begin{proof}
The proof is by induction on the height of $\mathcal D$, with cases on its last rule. An assumption $A$ with $\oa=\{A\}$ and $\ero=A$ translates to the initial sequent $A\seq A$.

\emph{Introduction rules.} These translate to the corresponding right rules. If the last rule is $\land\mathrm I$, with premisses derivations of $A$ from $\Gamma$ and of $B$ from $\Delta$, the induction hypothesis gives $\Gamma\seq A$ and $\Delta\seq B$, and $\land\mathrm R^{s}$ yields $\Gamma,\Delta\seq A\land B$. The rules $\lor\mathrm I_i$, $\arr\mathrm I_i$, $\lnot\lnot\mathrm I$, $\lnot\land\mathrm I_i$, $\lnot\lor\mathrm I$, $\lnot\arr\mathrm I$ translate to $\lor\mathrm R_i$, $\arr\mathrm R_i$, $\lnot\lnot\mathrm R$, $\lnot\land\mathrm R_i$, $\lnot\lor\mathrm R^{s}$, $\lnot\arr\mathrm R^{s}$ in the same manner, the two-premiss introductions using the independent-context right rules.

\emph{Explosion.} If the last rule is $\mathrm{Exp}$, with premisses derivations of $\lnot A$ from $\Gamma$ and of $A$ from $\Delta$, the induction hypothesis gives $\Gamma\seq\lnot A$ and $\Delta\seq A$, and $\lnot\mathrm{EX}$ yields $\Gamma,\Delta\seq C$.

\emph{General elimination rules.} These translate to the corresponding left rules, with a cut on the major premiss. Consider $\land\mathrm E$, with major premiss a derivation of $A\land B$ from $\Gamma$ and minor derivation of $C$ from $[A,B]$ and $\Delta$. The induction hypothesis gives $\Gamma\seq A\land B$ and $A,B,\Delta\seq C$; the latter, by $\land\mathrm L$, gives $A\land B,\Delta\seq C$; and a cut on $A\land B$ against $\Gamma\seq A\land B$ gives $\Gamma,\Delta\seq C$. The rules $\lor\mathrm E$, $\arr\mathrm E$, $\lnot\lnot\mathrm E$, $\lnot\land\mathrm E$, $\lnot\lor\mathrm E$, $\lnot\arr\mathrm E$ are treated in the same way, translating to $\lor\mathrm L^{s}$, $\arr\mathrm L^{s}$, $\lnot\lnot\mathrm L$, $\lnot\land\mathrm L^{s}$, $\lnot\lor\mathrm L$, $\lnot\arr\mathrm L$ respectively, each followed by a cut on the major premiss.

If $\mathcal D$ is fully normal, then the major premiss of every elimination is an assumption, so in each elimination case the derivation of the major premiss supplied by the induction hypothesis is an initial sequent $D\seq D$ with $D$ the principal formula; the cut against $D\seq D$ is trivial and may be omitted, the left rule being applied directly to the antecedent occurrence of $D$. Hence a fully normal $\mathcal D$ is simulated without cut. Since cut is admissible in $\Gz$ by Theorem~\ref{thm:cutG0}, in every case $\Gz\vdash\Gamma\seq C$.
\end{proof}

\subsection{From sequent calculus to natural deduction}\label{subsec:sctond}

The converse translation carries a cut-free derivation of $\Gz$ to a fully normal derivation of $\Ng$. It is here that the general format of the elimination rules earns its keep: a left rule, whose principal formula sits in the antecedent and is therefore an open assumption on the natural deduction side, translates to an elimination whose major premiss is exactly that assumption, so that the resulting derivation is fully normal by construction.

\begin{theorem}[Simulation of $\Gz$ in $\Ng$]\label{thm:sctond}
If $\Gz\vdash\Gamma\seq C$, then there is a fully normal derivation $\mathcal D$ in $\Ng$ with $\oa(\mathcal D)=\Gamma^{*}$, where $\Gamma^{*}$ is a multiset reduct of $\Gamma$, and $\ero(\mathcal D)=C$; a sequent $\Gamma\seq$ with empty succedent is translated with $\ero(\mathcal D)=\bot$.
\end{theorem}
\begin{proof}
By Theorem~\ref{thm:cutG0} we may assume the given $\Gz$-derivation to be cut-free. We argue by induction on its height, with cases on the last rule.

\emph{Initial sequents.} The sequent $A\seq A$ translates to the assumption $A$, a fully normal derivation with $\oa=\{A\}$ and $\ero=A$. The closure $A,\lnot A\seq$ translates to the application of $\mathrm{Exp}$ to the two assumptions $A$ and $\lnot A$, with conclusion $\bot$; both major premisses are assumptions, so the derivation is fully normal, with $\oa=\{A,\lnot A\}$ and $\ero=\bot$.

\emph{Structural rules.} Left weakening $\wkL$ adds an open assumption that is not discharged. Left contraction $\ctrL$ identifies two open assumptions of the same formula, passing to a multiset reduct of the antecedent. Succedent weakening $\wkR$, from $\Gamma\seq$ to $\Gamma\seq C$, is translated by applying the derived rule $\bot\mathrm E$ of Proposition~\ref{prop:falsum} to the translation of $\Gamma\seq$, whose end-formula is $\bot$; the result is fully normal.

\emph{One-premiss left rules.} If the last rule is $\land\mathrm L$, from $A,B,\Gamma\seq C$ to $A\land B,\Gamma\seq C$, the induction hypothesis gives a fully normal derivation of $C$ from $A,B$ and (a reduct of) $\Gamma$; applying $\land\mathrm E$ with major premiss the assumption $A\land B$ and this derivation as minor derivation, discharging $[A,B]$, yields a fully normal derivation of $C$ from $A\land B$ and (a reduct of) $\Gamma$, in which the new major premiss is again an assumption. The rules $\lnot\lnot\mathrm L$, $\lnot\lor\mathrm L$, $\lnot\arr\mathrm L$ translate to $\lnot\lnot\mathrm E$, $\lnot\lor\mathrm E$, $\lnot\arr\mathrm E$ in the same way.

\emph{Two-premiss left rules.} If the last rule is $\lor\mathrm L^{s}$, from $A,\Gamma\seq C$ and $B,\Delta\seq C$ to $A\lor B,\Gamma,\Delta\seq C$, the induction hypothesis gives fully normal derivations of $C$ from $A,\Gamma$ and from $B,\Delta$; applying $\lor\mathrm E$ with major premiss the assumption $A\lor B$ and these as the two minor derivations, discharging $[A]$ and $[B]$, yields a fully normal derivation of $C$ from $A\lor B,\Gamma,\Delta$. The rules $\lnot\land\mathrm L^{s}$ and $\arr\mathrm L^{s}$ translate to $\lnot\land\mathrm E$ and $\arr\mathrm E$ likewise.

\emph{Right rules and explosion.} Each right rule translates to the corresponding introduction: $\land\mathrm R^{s}$ to $\land\mathrm I$, $\lor\mathrm R_i$ to $\lor\mathrm I_i$, $\arr\mathrm R_i$ to $\arr\mathrm I_i$, $\lnot\lnot\mathrm R$ to $\lnot\lnot\mathrm I$, $\lnot\land\mathrm R_i$ to $\lnot\land\mathrm I_i$, $\lnot\lor\mathrm R^{s}$ to $\lnot\lor\mathrm I$, $\lnot\arr\mathrm R^{s}$ to $\lnot\arr\mathrm I$; the explosion rule $\lnot\mathrm{EX}$, from $\Gamma\seq\lnot A$ and $\Delta\seq A$ to $\Gamma,\Delta\seq C$, translates to $\mathrm{Exp}$ applied to the translations of the two premisses. In every case the introductions and $\mathrm{Exp}$ introduce no elimination whose major premiss is not an assumption, so full normality is preserved.

Since each rule of $\Gz$ is thus simulated by a step that preserves full normality, and the initial sequents translate to fully normal derivations, the derivation constructed is fully normal, with the stated open assumptions and end-formula.
\end{proof}

\subsection{Full normalization and equivalence}\label{subsec:norm}

The two translations combine into the full normalization theorem, in the manner of \cite[Thm.~3.18]{KamideNegri2025}. The essential point is that the passage through the sequent calculus launders an arbitrary derivation: a non-normal derivation is simulated in $\Gz$ with cuts, the cuts are eliminated by Theorem~\ref{thm:cutG0}, and the resulting cut-free derivation is translated back to a fully normal derivation.

\begin{theorem}[Full normalization for $\Ng$]\label{thm:norm}
For every derivation $\mathcal D$ in $\Ng$ with $\oa(\mathcal D)=\Gamma$ and $\ero(\mathcal D)=C$, there is a fully normal derivation $\mathcal D'$ in $\Ng$ with $\oa(\mathcal D')=\Gamma^{*}$, a multiset reduct of $\Gamma$, and $\ero(\mathcal D')=C$.
\end{theorem}
\begin{proof}
By Theorem~\ref{thm:ndtosc}, $\Gz\vdash\Gamma\seq C$. By Theorem~\ref{thm:cutG0}, this derivation may be taken cut-free. By Theorem~\ref{thm:sctond}, there is a fully normal derivation $\mathcal D'$ in $\Ng$ with $\oa(\mathcal D')$ a multiset reduct of $\Gamma$ and $\ero(\mathcal D')=C$, which was to be shown.
\end{proof}

\begin{theorem}[Equivalence of $\Ng$ and the sequent calculi]\label{thm:ndequiv}
For every finite multiset $\Gamma$ and every formula $A$, the following are equivalent: \textup{(i)} there is a derivation of $A$ in $\Ng$ whose open assumptions form a multiset reduct of $\Gamma$; \textup{(ii)} $\Gz\vdash\Gamma\seq A$; \textup{(iii)} $\Gt\vdash\Gamma\seq A$. When these hold, the derivation in \textup{(i)} may be taken fully normal.
\end{theorem}
\begin{proof}
The equivalence of (ii) and (iii) is Theorem~\ref{thm:equiv}. Suppose (i), with open assumptions $\Gamma^{*}$ a multiset reduct of $\Gamma$. By Theorem~\ref{thm:ndtosc}, $\Gz\vdash\Gamma^{*}\seq A$, and repeated applications of $\wkL$ restore the lowered multiplicities, giving $\Gz\vdash\Gamma\seq A$, which is (ii). Suppose (ii). By Theorem~\ref{thm:sctond} there is a fully normal derivation of $A$ in $\Ng$ whose open assumptions form a multiset reduct of $\Gamma$, which is (i) in the stronger form announced.
\end{proof}

\begin{corollary}\label{cor:notheorems}
No derivation in $\Ng$ has all of its assumptions discharged; that is, $\Ng$ has no theorems.
\end{corollary}
\begin{proof}
Such a derivation would give, by Theorem~\ref{thm:ndtosc}, a derivation of $\seq A$ in $\Gz$, which Corollary~\ref{cor:noemptyG0} excludes. The system $\Ng$ is a calculus of consequence, and Theorem~\ref{thm:ndequiv} states the sense in which it agrees with the sequent calculi; for the classical theorems one passes to $\Gt^{c}$ and to the natural deduction system with the classical reductio, as in Remark~\ref{rem:raa}.
\end{proof}

The normalization theorem for $\Ng$ and the cut-elimination theorem for $\Gz$ are, in this way, two aspects of a single structural fact, made to correspond by the bidirectional translation: a fully normal derivation of natural deduction and a cut-free derivation of the sequent calculus are the images of one another, and the discipline that every major premiss of an elimination be an assumption is the natural deduction counterpart of the analyticity of the cut-free sequent calculus, itself the counterpart of the cut-freeness of the block calculus with which we began.

\section{Concluding remarks}\label{sec:remarks}

The development of the previous sections places the calculus of sequent-style tableaux within the same proof-theoretic architecture that Kamide and Negri erected for the extended Belnap--Dunn and intuitionistic logics. Starting from the block calculus, whose only ambition was to combine the readability of the sequent notation with the refutational efficiency of tableaux, we have obtained a structural-rule-free G3-style calculus $\Gt$ whose empty-succedent fragment is the block calculus itself, a G0-style calculus $\Gz$ with independent contexts and explicit structural rules, a theorem establishing the equivalence of the two, the cut-elimination theorem for $\Gz$ as a consequence of that equivalence, a natural deduction system $\Ng$ with general elimination rules, and a full normalization theorem for $\Ng$ obtained by bidirectional translation. The classical strength of the calculus is carried by the refutational reading, and by it alone: the closable blocks --- equivalently, the sequents derivable with empty succedent --- are exactly the classically unsatisfiable finite sets of formulae, and it is on this fragment that the involutive De Morgan negation delivers classical logic in full. The formula-succedent sequents and the natural deduction derivations that correspond to them furnish, above this base, the constructive scaffolding on which the normalization theorem rests; they do not of themselves furnish theorems, as Proposition~\ref{prop:noempty} records, and in this the present systems differ from the intuitionistic-based systems of Kamide and Negri, whose succedent side is the logic. Adjoining the rule of indirect proof restores classical logic on the succedent side as well, and restores it exactly (Proposition~\ref{prop:raa}). Figure~\ref{fig:summary} records the resulting network of systems and the theorems that connect them.

\begin{figure}[htb]
\centering
\begin{tikzpicture}[
  >={Latex[length=2.4mm]},
  box/.style={draw, rounded corners=2pt, align=center, inner sep=6pt,
              minimum height=13mm, text width=52mm, line width=0.5pt},
  smallbox/.style={draw, rounded corners=2pt, align=center, inner sep=5pt,
              minimum height=10mm, text width=38mm, line width=0.45pt},
  elab/.style={font=\footnotesize, align=left, inner sep=3pt, text width=44mm},
  prop/.style={font=\footnotesize\itshape, align=center, text width=24mm},
  eq/.style={<->, line width=0.7pt, shorten >=2pt, shorten <=2pt},
  ar/.style={->, line width=0.7pt, shorten >=2pt, shorten <=2pt},
  tie/.style={line width=0.3pt, densely dotted, shorten >=2pt, shorten <=2pt},
  vv/.style={<->, line width=0.45pt, densely dashed, shorten >=2pt, shorten <=2pt},
]
\node[smallbox] (lk) at (0,9.9) {\small\textbf{cut-free $\LK$}\\[1pt]\footnotesize Gentzen};
\node[box]      (bt) at (0,8.0) {\small\textbf{block calculus}\\[1pt]\footnotesize sequent-style tableaux};
\node[box]      (g3) at (0,5.6) {\small\textbf{$\Gt$} \ \footnotesize (G3-style)\\[2pt]\footnotesize shared contexts,\\ no structural rules};
\node[box]      (g0) at (0,2.8) {\small\textbf{$\Gz$} \ \footnotesize (G0-style)\\[2pt]\footnotesize split contexts,\\ explicit wk, ctr};
\node[box]      (ng) at (0,0.0) {\small\textbf{$\Ng$} \ \footnotesize (natural deduction)\\[2pt]\footnotesize general elimination rules};
\draw[eq] (lk) -- node[font=\footnotesize, right=2pt]{$\cong$ \ \cite{Cuconato2025IM}} (bt);
\draw[ar] (bt) -- node[elab, right]{Prop.~\ref{prop:blockfragment}:\\ the $\Pi\seq$ fragment} (g3);
\draw[eq] (g3) -- node[elab, right]{Thm.~\ref{thm:equiv} (equivalence);\\ Thm.~\ref{thm:cutG0}: cut elimination} (g0);
\draw[eq] (g0) -- node[elab, right]{Thms.~\ref{thm:ndtosc},~\ref{thm:sctond} (translations);\\ Thm.~\ref{thm:norm}: full normalization} (ng);
\node[prop] (p1) at (-4.5,5.6) {analyticity};
\node[prop] (p2) at (-4.5,2.8) {cut-freeness};
\node[prop] (p3) at (-4.5,0.0) {full normal form};
\draw[tie] (p1) -- (g3);
\draw[tie] (p2) -- (g0);
\draw[tie] (p3) -- (ng);
\draw[vv] (p1) -- (p2);
\draw[vv] (p2) -- (p3);
\end{tikzpicture}
\caption{\rm Proof-theoretic network induced by the sequent-style tableaux. The block calculus is the empty-succedent fragment of the cut-free sequent calculus $\Gt$; $\Gt$ and $\Gz$ are theorem-equivalent, and $\Gz$ is equivalent to the natural deduction system $\Ng$. The labels on the left indicate the proof-theoretic property realized at each level.}
\label{fig:summary}
\end{figure}
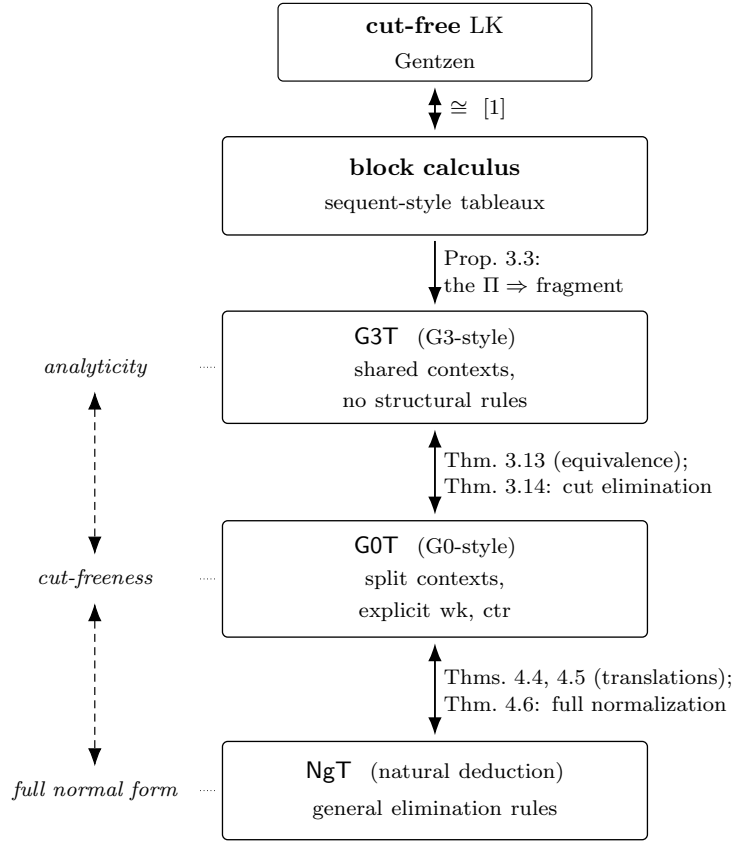

Two directions of extension deserve mention. The first is to the propositional neighbours of the calculus. The choice of an involutive De Morgan negation and of a material implication fixes the classical logic; weakening the negation rules --- dropping the involution $\lnot\lnot\mathrm L$, $\lnot\lnot\mathrm R$, or the closure on complementary pairs --- yields the subsystems of the Belnap--Dunn family, for which the same three-fold correspondence can be set up along the lines traced here, and which we do not pursue. The second, and more substantial, is the extension to first-order logic, which we state as a remark since it is the natural continuation of the present work.

\begin{remark}[The first-order extension]\label{rem:firstorder}
The first-order block calculus of \cite{Cuconato2025IM} adds to the propositional rules the universal rules of type $\gamma$, which instantiate a universal formula on a parameter of the branch while retaining the principal formula, and the existential rules of type $\delta$, which instantiate on a parameter fresh for the branch and consume the principal formula. In the sequent presentation these become the quantifier rules of $\Gt$ and $\Gz$: a rule for $\forall$ on the left and for $\lnot\exists$ on the left with an arbitrary instantiating term, corresponding to the $\gamma$ rules, and a rule for $\exists$ on the left and $\lnot\forall$ on the left with a proper parameter --- the eigenvariable --- corresponding to the $\delta$ rules. The freshness condition on the fresh constant $a_{\uparrow}$ of the block calculus is, on the sequent side, exactly the eigenvariable condition, and on the natural deduction side exactly the restriction on the parameter of the general elimination rule for the existential quantifier; the structural analogy between the necessitation restriction of modal logic and the eigenvariable condition, which Negri makes in the modal setting, has here its quantificational instance. The general elimination rules for the quantifiers take the form
\[
\frac{\forall x\,A\quad \begin{array}{c}[A(t)]\\ \vdots\\ C\end{array}}{C}
\qquad
\frac{\exists x\,A\quad \begin{array}{c}[A(a)]\\ \vdots\\ C\end{array}}{C}
\]
with $t$ an arbitrary term in the first and $a$ a parameter, proper to the rule, in the second, together with the De Morgan forms for $\lnot\forall$ and $\lnot\exists$. The equivalence theorem, the cut-elimination theorem and the full normalization theorem extend to the first-order calculi by adding to the inductions of Section~\ref{sec:sequents} and Section~\ref{sec:nd} the cases of the quantifier rules, in which the only new ingredient is the systematic renaming of proper parameters that the substitution and the eigenvariable condition require --- exactly the renaming that the first-order correspondence with $\LK$ already demands in \cite{Cuconato2025IM}. We refrain from carrying out these inductions here, since they add length without altering the structure of the argument, and since the propositional case already exhibits the whole of the proof-theoretic mechanism.
\end{remark}

It is worth closing on the methodological point that motivates the whole construction. The block calculus was designed as a calculus of \emph{analysis}: one begins with the negation of what is to be proved and decomposes, and the closure of every branch is the proof. The sequent calculus is a calculus of \emph{synthesis}: one composes the rules upward to the endsequent. The natural deduction with general elimination rules occupies a middle ground, in which the eliminations decompose while the introductions compose, and full normal form is precisely the regime in which the two do not interfere. That these three readings of one and the same combinatorial structure can be translated into one another, and that the translation carries analyticity to cut-freeness to full normality, is the sense in which the calculus of sequent-style tableaux, modest in its origin, is a full member of the structural theory of proof.


\end{document}